\documentclass[11pt, leqno]{article}

\usepackage{amsmath,amssymb,amsthm, epsfig}

\usepackage[numbers]{natbib}

\usepackage{hyperref}
\usepackage[nameinlink]{cleveref}
\hypersetup{colorlinks=true,linkcolor=blue,citecolor=blue,urlcolor=blue}

\usepackage{amsmath}
\theoremstyle{plain}
\theoremstyle{definition}
\allowdisplaybreaks
\usepackage[pdftex]{color}
\usepackage{comment}

\usepackage{mathtools}
\usepackage{color}

\usepackage{ulem}

\usepackage{mathrsfs}

\usepackage{wasysym}

\usepackage{stackrel}

\usepackage{amsthm} 

\newcommand{\mint}{-\; \!\! \!\!\!\!\!\!\int}
\newtheorem{assumptionBase}{\hspace{-5pt}} 

\usepackage{pgfplots}
\pgfplotsset{compat=newest}

\title{\large{\bf Bernoulli problem for the fractional $p$-Laplacian}}
\author{Rafael R. Costa\footnote{\noindent Universidade Federal do Rio Grande do Sul, Porto Alegre-RS, Brazil. {E-mail: \url{ramos.costa@ufrgs.br}}} \and Diego Marcon\footnote{\noindent Universidade Federal do Rio Grande do Sul, Porto Alegre-RS, Brazil. {E-mail: \url{diego.marcon@ufrgs.br}}}}

\newlength{\hchng}
\newlength{\vchng}
\newcommand{\dd}{\mathrm{d}}

\newtheorem{theorem}{Theorem}[section]

\newtheorem{lemma}[theorem]{Lemma}

\newtheorem{proposition}[theorem]{Proposition}

\newtheorem{corollary}[theorem]{Corollary}

\theoremstyle{definition}

\newtheorem{definition}[theorem]{Definition}

\theoremstyle{remark}

\newtheorem{remark}[theorem]{Remark}

\numberwithin{equation}{section}

\newcommand{\intav}[1]{\mathchoice {\mathop{\vrule width 6pt height 3 pt depth  -2.5pt
\kern -8pt \intop}\nolimits_{\kern -6pt#1}} {\mathop{\vrule width
5pt height 3  pt depth -2.6pt \kern -6pt \intop}\nolimits_{#1}}
{\mathop{\vrule width 5pt height 3 pt depth -2.6pt \kern -6pt
\intop}\nolimits_{#1}} {\mathop{\vrule width 5pt height 3 pt depth
-2.6pt \kern -6pt \intop}\nolimits_{#1}}}

\begin{document}

\maketitle
\begin{abstract}
		We study regularity properties of minimizers for the one-phase Alt--Caffarelli problem associated with the fractional $p$-Laplacian in the range $p\geq2$. We consider minimizers of the fractional $p$-energy penalized by the measure of the positivity set, with prescribed nonnegative exterior datum. We prove existence of minimizers and derive their basic variational properties: minimizers are nonnegative and are weak subsolutions of the homogeneous fractional $p$-Laplace equation. The main regularity argument combines fractional $p$-harmonic replacements, energy-gap estimates, nonlocal tail bounds, and a Campanato-type iteration. This yields local Hölder continuity of minimizers and implies that they solve the homogeneous equation in their positivity set. Finally, we prove the optimal free boundary growth estimate, showing that the sharp order known in the linear fractional Bernoulli problem persists in the fractional $p$-Laplacian setting.

		\medskip
		\noindent \textbf{Keywords}: Alt--Caffarelli problem;
		Fractional $p$-Laplacian;
		Free boundary problem;
		Hölder regularity;
		Optimal growth.
		\vspace{0.2cm}
		
		\noindent \textbf{2020 Mathematics Subject Classification:} 35R35, 35R11, 35J92, 35B65, 49J40.
	\end{abstract}


	\section{Introduction}
	
	We study a one-phase Alt--Caffarelli type free boundary problem, also known as one-phase Bernoulli problem, driven by
	the fractional $p$-Laplacian with $p\geq2$. Let $\Omega\subset\mathbb{R}^n$ be a bounded
	domain, $s\in(0,1)$, and $M>0$. Given a nonnegative exterior
	datum $g\in \mathrm{W}^{s,p}(\mathbb{R}^n)$, we consider minimizers of
	\begin{equation}\label{main functional}
	\mathcal{I}_M(u)
	:= \int_{\mathbb{R}^n}\int_{\mathbb{R}^n}
	\frac{|u(x)-u(y)|^p}{|x-y|^{n+sp}}
	\,\dd x\,\dd y
	+
	M|\{u>0\}\cap\Omega|    
	\end{equation}
	over the admissible class
	\begin{equation}\label{minimizer}
	    \mathcal A_g(\Omega)
	:=
	\left\{
	v\in \mathrm{W}^{s,p}(\mathbb{R}^n) \, ; ~
	v=g \ \text{a.e. in } \mathbb{R}^n\setminus\Omega
	\right\}.
	\end{equation}
	Equivalently, since the exterior datum $g$ is fixed, one may subtract from the
	energy the fixed interaction over $\Omega^c\times\Omega^c$ and work with
	the integral over
	$\mathbb{R}^{2n}\setminus(\Omega^c\times\Omega^c)$. Both formulations
	give the same minimizers.
	
	The local one-phase Alt--Caffarelli
	problem originates in the seminal work \cite{AC81}, where Alt and Caffarelli develop the regularity theory for minimizers of the Dirichlet
	energy penalized by the measure of the positivity set. The two-phase
	version is treated by Alt, Caffarelli, and Friedman in \cite{ACF84}.
	For quasilinear diffusion, Danielli and Petrosyan study in \cite{DP01}
	the corresponding problem for the $p$-Laplacian, associated with the
	functional
	\[
	J(u)
	=
	\int_{\Omega}|\nabla u|^p\,\dd x
	+
	\lambda^p|\{u>0\}\cap\Omega|.
	\]
	In the positivity set, minimizers are $p$-harmonic, while on the free
	boundary one expects a Bernoulli-type condition. The local theory is
	extensive; we refer to \cite{CS05,D23,V23} for background, further
	developments, and geometric aspects of the free boundary.
	
	The nonlocal one-phase Bernoulli problem is introduced by Caffarelli,
	Roquejoffre, and Sire in \cite{CRS10} for the fractional Laplacian. In
	that case, the model energy is
	\[
	\mathcal J(u)
	=
	\int_{\mathbb{R}^n}\int_{\mathbb{R}^n}
	\frac{|u(x)-u(y)|^2}{|x-y|^{n+2s}}
	\,\dd x\,\dd y
	+
	|\{u>0\}\cap\Omega|.
	\]
	The nonlocal character of the energy creates a fundamental difficulty:
	local estimates depend not only on values of the function near the point
	under consideration, but also on values far away. This influence is
	measured by nonlocal tails. In the fractional $p$-Laplacian setting, the
	tail is given by
	\[
	\operatorname{Tail}(u;x_0,r)
	:=
	\left(
	r^{sp}
	\int_{\mathbb{R}^n\setminus B_r(x_0)}
	\frac{|u(y)|^{p-1}}{|y-x_0|^{n+sp}}
	\,\dd y
	\right)^{1/(p-1)}.
	\]

The work \cite{CRS10} uses the Caffarelli--Silvestre extension
\cite{CS07}, which transforms the fractional Laplacian into a local
degenerate elliptic operator in one higher dimension. This approach leads
to optimal regularity, nondegeneracy, and the classification of global
solutions. In dimension two, it also gives $\mathrm{C}^{1}$ regularity of Lipschitz
free boundaries. In a related optimal-design direction, Teixeira and
Teymurazyan \cite{TT15} use the extension method to study fractional
diffusion problems with volume constraints, proving existence, optimal
regularity, and free-boundary regularity outside a negligible singular set. Back to the fractional Bernoulli problem, the free boundary regularity theory is developed further, first
for $s=1/2$ and then for the full range $s\in(0,1)$; see
\cite{DR12,DS12,DS15b,DSS14,DS15a,EKPSS21}.

	More recently, Ros-Oton and Weidner \cite{RO24,RW27} and Snelson and
	Teixeira \cite{ST25} study nonlocal Alt--Caffarelli type problems for
	more general linear integro-differential operators. In this setting, the
	kernels need not be translation invariant, and the Caffarelli--Silvestre
	extension is no longer available. Thus, the analysis must address the
	nonlocal structure directly. Ros-Oton and Weidner prove regularity results
	for minimizers and apply them to boundary $\mathrm{C}^{s}$ estimates for nonlocal
	elliptic equations with kernels comparable to $|y|^{-n-2s}$; they also
	obtain $\mathrm{C}^{1,\alpha}$ regularity of the free boundary near regular points.
	Snelson and Teixeira treat a broader class of kernels and establish local
	Hölder continuity, optimal regularity along the free boundary, and
	nondegeneracy by combining scaling arguments with the regularity theory
	for the corresponding nonlocal equations.
	
	Compared with the linear fractional case, the free boundary theory for
	the fractional $p$-Laplacian is still much less developed. Related
	nonlinear nonlocal problems appear, for instance, in the obstacle problem
	for the fractional $p$-Laplacian \cite{KKP16b,KKP16a} and in the
	$p$-fractional optimal design problem studied in \cite{SR19}. To the best
	of our knowledge, however, the Alt--Caffarelli free boundary problem
	associated with the fractional $p$-Laplacian has not been systematically
	developed. The purpose of this paper is to initiate this study by proving
	existence, basic variational properties, local Hölder regularity, and the
	optimal growth of minimizers at free boundary points.
	
	We now state the main results. The first one gives the basic variational
	framework.
	
	\begin{theorem}[Existence and first properties]\label{thm:intro-existence}
		Let $\Omega\subset\mathbb{R}^n$ be a bounded domain, let $s\in(0,1)$, and let $g\in \mathrm{W}^{s,p}(\mathbb{R}^n)$ satisfy
        $g\geq0$ a.e{.} in $\mathbb{R}^n\setminus\Omega$. Then, there exists a
		minimizer of $\mathcal I_M$ in $\mathcal A_g(\Omega)$. Moreover, every
		minimizer $u$ satisfies $u\geq0$ a.e. in $\mathbb{R}^n$ and is a weak
		subsolution of the homogeneous fractional $p$-Laplace equation in
		$\Omega$.
	\end{theorem}
	
	The subsolution property is the first link between the variational free
	boundary problem and the regularity theory for the fractional
	$p$-Laplacian. In particular, it yields local boundedness of minimizers.
	Our next main result proves local Hölder continuity. The proof compares a
	minimizer with its fractional $p$-harmonic replacement, controls the
	energy gap through the volume term, and then applies a Campanato-type
	iteration.
	
	\begin{theorem}[Local Hölder regularity]\label{thm:intro-holder}
		Let $u$ be a minimizer of $\mathcal I_M$ in $\Omega$. Then, there exist
		$\alpha_0\in(0,1)$ and $c>0$, depending only on $n,p$, and $s$, such
		that, for every ball $B_{4R}(x_0)\subset\Omega$,
		\[
		[u]_{\mathrm{C}^{\alpha_0}(B_R(x_0))}
		\leq
		\frac{c}{R^{\alpha_0}}
		\left[
		\operatorname{Tail}(u;x_0,2R)
		+
		\left(
		\mint_{B_{4R}(x_0)}
		|u|^p\,\dd x
		\right)^{1/p}
		+
		M^{1/p}R^s
		\right].
		\]
	\end{theorem}
	
	Once minimizers are continuous, the positivity set
	$\Omega\cap\{u>0\}$ is open. Therefore, variations compactly supported
	inside the positivity set do not change the volume term, and the usual
	first variation of the fractional energy gives the homogeneous equation.
	
	\begin{corollary}\label{thm:intro-positive-phase}
		Let $u$ be a minimizer of $\mathcal I_M$ in $\Omega$. Then, $(-\Delta)^s_p u=0$ weakly in $\Omega\cap\{u>0\}$.
	\end{corollary}
	
	The final main result concerns the behavior of minimizers near free
	boundary points. In the fractional Laplacian case $p=2$, the optimal
	growth at the free boundary is of order $r^s$. We prove that the same
	growth rate holds for the fractional $p$-Laplacian.
	
	\begin{theorem}[Optimal growth at free boundary points]\label{thm:intro-growth}
		Let $u$ be a nonnegative minimizer of $\mathcal I_M$ in $\Omega$. Then,
		there exist constants $c>0$ and $\widetilde r>0$, with $c$ depending only
		on $n,s$, and $p$, and $\widetilde r$ depending only on $n,s,p$, and $M$,
		such that the following holds. If $x_0\in\partial\{u>0\}$,
		$0<r<\widetilde r$, and $B_{2r}(x_0)\subset\Omega$; then, for every
		$x\in B_{r/2}(x_0)$,
		\[
		u(x)
		\leq
		\dfrac{c}{r^{s}}
		\left[
		\operatorname{Tail}(u;x_0,r)
		+
		\left(
		\mint_{B_{2r}(x_0)}u^p\,\dd x
		\right)^{1/p}
		+
		1
		\right]
		|x-x_0|^s.
		\]
	\end{theorem}
	
	The proof of \Cref{thm:intro-growth} follows an iteration scheme inspired
	by the strategy of \cite{ST25}, but it requires estimates adapted to the
	nonlinear fractional $p$-Laplacian. The main step is a flatness lemma for
	small values of the penalization parameter. The iteration propagates both
	a supremum decay and a compatible tail bound across dyadic scales, which
	produces the sharp exponent $s$.

    The paper is organized as follows. In \Cref{sec:preliminaries}, we fix the notation, recall the fractional Sobolev framework, and introduce the weak formulation of the fractional $p$-Laplacian. We also collect the local
boundedness, Hölder, Harnack, and tail estimates used throughout the paper,
and prove the auxiliary estimates for fractional $p$-harmonic replacements.
In \Cref{sec:existence-first-properties}, we prove the existence of
minimizers and establish their first variational properties, including
nonnegativity, the subsolution property, and local boundedness. In
\Cref{Sec Holder regularity}, we prove the local Hölder estimate stated in
\Cref{thm:intro-holder} by comparing minimizers with their fractional
$p$-harmonic replacements and applying a Campanato argument. We also show
that minimizers solve the homogeneous fractional $p$-Laplace equation in
their positivity set, proving \Cref{thm:intro-positive-phase}. Finally, in
\Cref{sec:optimal-growth}, we prove the flatness lemma and use it in an
iteration scheme to establish the optimal free boundary growth estimate
\Cref{thm:intro-growth}.

\section{Preliminaries}\label{sec:preliminaries}

		In this section, we fix notation, introduce the fractional Sobolev spaces and the
		weak formulation of the fractional $p$-Laplacian, and collect the nonlocal estimates
		we use throughout the paper. We also record the elementary variational consequence
		that minimizers of $\mathcal I_M$ are weak subsolutions of the homogeneous
		fractional $p$-Laplace equation.
		
		Throughout the paper, $\Omega\subset\mathbb{R}^n$ is a bounded domain,
		$s\in(0,1)$ and $p\ge 2$. We write
		\[
		u_+ := \max\{u,0\},
		\quad
		u_- := \max\{-u,0\},
		\]
		and, whenever $E\subset\mathbb{R}^n$ has positive and finite measure,
		\[
		(u)_E := \mint_E u\,\dd x
		=
		\frac{1}{|E|}\int_E u\,\dd x.
		\]
		In particular, if $E=B_r(x_0)$, we also write $(u)_{x_0,r}$ or simply
		$(u)_r$ when the center is clear from the context.
		
		We denote by $\mathrm{W}^{s,p}(\mathbb{R}^n)$ the fractional Sobolev space
		\[
		\mathrm{W}^{s,p}(\mathbb{R}^n)
		:=
		\left\{
		u\in \mathrm{L}^p(\mathbb{R}^n)
		\, ; ~
		[u]_{\mathrm{W}^{s,p}(\mathbb{R}^n)}<\infty
		\right\},
		\]
		where
		\[
		[u]_{\mathrm{W}^{s,p}(\mathbb{R}^n)}^p
		:=
		\int_{\mathbb{R}^n}\int_{\mathbb{R}^n}
		\frac{|u(x)-u(y)|^p}{|x-y|^{n+sp}}
		\,\dd x\,\dd y.
		\]
		The norm is given by
		\[
		\|u\|_{\mathrm{W}^{s,p}(\mathbb{R}^n)}
		:=
		\left(
		\|u\|_{\mathrm{L}^p(\mathbb{R}^n)}^p
		+
		[u]_{\mathrm{W}^{s,p}(\mathbb{R}^n)}^p
		\right)^{1/p}.
		\]
		We also use the notation
		\[
		\mathrm{W}^{s,p}_0(\Omega)
		:=
		\left\{
		u\in \mathrm{W}^{s,p}(\mathbb{R}^n)
		:
		u=0 \text{ a.e. in } \mathbb{R}^n\setminus\Omega
		\right\}.
		\]
		Thus, given an exterior datum $g\in \mathrm{W}^{s,p}(\mathbb{R}^n)$, the natural
		admissible class is
		\[
		\mathcal A_g(\Omega)
		:=
		\left\{
		u\in \mathrm{W}^{s,p}(\mathbb{R}^n)
		:
		u-g\in \mathrm{W}^{s,p}_0(\Omega)
		\right\}.
		\]
		Equivalently, as in \eqref{minimizer}, $u=g$ a.e. in $\mathbb{R}^n\setminus\Omega$.
		
		For $v\in \mathrm{W}^{s,p}(\mathbb{R}^n)$, $x_0\in\mathbb{R}^n$ and $r>0$, we
		define the nonlocal tail by
		\[
		\operatorname{Tail}(v;x_0,r)
		:=
		\left(
		r^{sp}
		\int_{\mathbb{R}^n\setminus B_r(x_0)}
		\frac{|v(y)|^{p-1}}{|y-x_0|^{n+sp}}
		\,\dd y
		\right)^{1/(p-1)}.
		\]
        We also introduce the associated tail space: for $p>1$ and $s\in(0,1)$, we set
\[
\mathrm{L}^{p-1}_{sp}(\mathbb{R}^n)
:=
\left\{
u\in \mathrm{L}^{p-1}_{\operatorname{loc}}(\mathbb{R}^n)
:
\operatorname{Tail}(u;z,r)<\infty
\ \text{for every } z\in\mathbb{R}^n
\text{ and every } r>0
\right\}.
\]
This space can be equivalently characterized as
\[
\mathrm{L}^{p-1}_{sp}(\mathbb{R}^n)
=
\left\{
u\in \mathrm{L}^{p-1}_{\operatorname{loc}}(\mathbb{R}^n)
:
\int_{\mathbb{R}^n}
\frac{|u(x)|^{p-1}}{(1+|x|)^{n+sp}}
\,\dd x
<\infty
\right\}.
\]
		
		We now introduce the weak formulation of the fractional $p$-Laplacian. 
		For $u,\varphi\in \mathrm{W}^{s,p}(\mathbb{R}^n)$, set
		\[
		\mathcal E(u,\varphi)
		:= 
		\int_{\mathbb{R}^n} \int_{\mathbb{R}^n}
		\frac{
			j_p(u(x)-u(y))(\varphi(x)-\varphi(y))
		}
		{|x-y|^{n+sp}}
		\,\dd x\,\dd y,
		\] where $j_p(t):=|t|^{p-2}t$.
		Formally, the corresponding operator is
		\begin{equation}\label{fractional-p-laplacian}
			\mathcal L u(x)
			:=
			\operatorname{P.V.}
			\int_{\mathbb{R}^n}
			\frac{
				|u(x)-u(y)|^{p-2}(u(x)-u(y))
			}
			{|x-y|^{n+sp}}
			\,\dd y.
		\end{equation}
		The weak formulation is independent of the multiplicative constants that
		may appear when passing from the double-integral form to the principal value form.
		
		Let $D\subset\Omega$ be open. We say that $u\in\mathrm{W}^{s,p}(\mathbb{R}^n)$ is
		a weak subsolution of
		\[
		\mathcal L u=0
		\quad\text{in } D
		\]
		if $\mathcal E(u,\varphi)\leq0$, for every nonnegative $\varphi\in \mathrm{W}^{s,p}_0(D)$. Similarly, $u$ is a weak
		supersolution if $\mathcal E(u,\varphi)\geq0$, for every nonnegative $\varphi\in \mathrm{W}^{s,p}_0(D)$. A weak solution is both a
		weak subsolution and a weak supersolution, that is,
		\[
		\mathcal E(u,\varphi)=0,
		\]
		for every $\varphi\in \mathrm{W}^{s,p}_0(D)$.
		Moreover, given $g\in \mathrm{W}^{s,p}(\mathbb{R}^n)$, we say that $u$ solves
		\begin{equation}\label{maineq}
			\left\{
			\begin{array}{rcl}
				\mathcal L u&=&0
				\quad\text{in } \Omega,\\
				u&=&g
				\quad\text{in } \mathbb{R}^n\setminus\Omega,
			\end{array}
			\right.
		\end{equation}
		in the weak sense if $u\in\mathcal A_g(\Omega)$ and $\mathcal E(u,\varphi)=0$, for every $\varphi\in \mathrm{W}^{s,p}_0(\Omega)$.
		
		We now define minimizers of the free boundary functional $\mathcal I_M$ introduced
		in \eqref{main functional}.
			Let $M>0$ and let $g\in \mathrm{W}^{s,p}(\mathbb{R}^n)$. We say that
			$u\in\mathcal A_g(\Omega)$ is a minimizer of $\mathcal I_M$ in $\Omega$ with
			exterior datum $g$ if, for every $v\in\mathcal A_g(\Omega)$,
			\[
			\mathcal I_M(u)\leq \mathcal I_M(v).
			\]

		\subsection{Estimates for fractional \texorpdfstring{$p$}{p}-harmonic functions}
		
		We now collect the nonlocal estimates used in the regularity arguments. The first is the local boundedness estimate for weak subsolutions; see
		\cite[Theorem 1.1]{CKP16}.
		
		\begin{theorem}[Local boundedness]\label{Boundedness}
			Let $p\in(1,+\infty)$ and let $u\in\mathrm{W}^{s,p}(\mathbb{R}^n)$ be a weak
			subsolution of $\mathcal L u=0$ in $\Omega$. If $B_r(x_0)\subset\Omega$; then,
			\[
			\sup_{B_{r/2}(x_0)}u
			\leq
			\delta\operatorname{Tail}(u_+;x_0,r/2)
			+
			c\,\delta^{-\frac{(p-1)n}{sp^2}}
			\left(
			\mint_{B_r(x_0)}u_+^p\,\dd x
			\right)^{1/p}
			\]
			for every $\delta\in(0,1]$, where $c>0$ depends only on $n,p$, and $s$.
		\end{theorem}
		
		The next result is the local Hölder estimate for weak solutions; see
		\cite[Theorem 1.2]{CKP16}.
		
		\begin{theorem}[Local Hölder estimate]\label{CKP}
			Let $p\in(1,+\infty)$ and let $u\in\mathrm{W}^{s,p}(\mathbb{R}^n)$ be a weak
			solution of $\mathcal L u=0$ in $\Omega$. Then, $u$ is locally Hölder continuous in
			$\Omega$. More precisely, there exist $\alpha\in(0,1)$, with
			$\alpha<sp/(p-1)$, and $c>0$, both depending only on $n,p$, and $s$, such that,
			if $B_{2R}(x_0)\subset\Omega$, then
			\[
			\operatorname{osc}_{B_r(x_0)}u
			\leq
			c
			\left(\frac{r}{R}\right)^\alpha
			\left[
			\operatorname{Tail}(u;x_0,R)
			+
			\left(
			\mint_{B_{2R}(x_0)}|u|^p\,\dd x
			\right)^{1/p}
			\right],
			\]
			for every $r\in(0,R]$.
		\end{theorem}
		
		As a direct consequence, we obtain the following Campanato-type estimate.
		
		\begin{corollary}\label{campanato_estimate}
			Let $u\in\mathrm{W}^{s,p}(\mathbb{R}^n)$ be a weak solution of
			$\mathcal L u=0$ in $\Omega$ and let $\alpha\in(0,1)$ be the exponent from \Cref{CKP}. If $B_{2R}(x_0)\subset\Omega$, then, for every
			$r\in(0,R]$,
			\[
			\int_{B_r(x_0)}
			|u-(u)_{x_0,r}|^p
			\,\dd x
			\leq
			c r^n
			\left(\frac{r}{R}\right)^{p\alpha}
			\left[
			\operatorname{Tail}(u;x_0,R)
			+
			\left(
			\mint_{B_{2R}(x_0)}|u|^p\,\dd x
			\right)^{1/p}
			\right]^p,
			\]
			where $c>0$ depends only on $n,p$, and $s$.
		\end{corollary}
		
		
		For the optimal growth analysis near free boundary points, we also use a
		nonlocal Harnack inequality \cite[Theorem 1.1]{CKP14} and a corresponding tail estimate	\cite[Lemma 4.2]{CKP14}.
		
		\begin{theorem}[Nonlocal Harnack inequality]\label{Nonlocal Harnack inequality}
			For any $s\in(0,1)$ and $p\in(1,+\infty)$, let
			$u\in\mathrm{W}^{s,p}(\mathbb{R}^n)$ be a weak solution of $\mathcal L u=0$ in
			$\Omega$. Assume that $u\geq0$ in $B_R(x_0)\subset\Omega$. Then, for every $0<r\leq R/2$,
			\[
			\sup_{B_r(x_0)}u
			\leq
			c\inf_{B_r(x_0)}u
			+
			c
			\left(\frac{r}{R}\right)^{\frac{sp}{p-1}}
			\operatorname{Tail}(u_-;x_0,R),
			\]
			where $c>0$ depends only on $n,p$, and $s$.
		\end{theorem}
		
		\begin{proposition}[Tail estimate]\label{Tail estimate01}
			Let $s\in(0,1)$, $p\in(1,+\infty)$, and let
			$u\in\mathrm{W}^{s,p}(\mathbb{R}^n)$ be a weak solution of $\mathcal L u=0$ in
			$\Omega$. Assume $u\geq0$ in $B_R(x_0)\subset\Omega$. Then, for every $0<r<R$,
			\[
			\operatorname{Tail}(u_+;x_0,r)
			\leq
			c\sup_{B_r(x_0)}u
			+
			c
			\left(\frac{r}{R}\right)^{\frac{sp}{p-1}}
			\operatorname{Tail}(u_-;x_0,R),
			\]
			where $c>0$ depends only on $n,p$ and $s$.
		\end{proposition}

\subsection{Auxiliary estimates}

We now collect some auxiliary estimates that are used to prove the local
Hölder regularity of minimizers of \eqref{main functional}. We begin by
introducing the fractional \(p\)-harmonic replacement in the setting of the
present paper. This replacement provides a natural comparison map for the
fractional energy, while the additional estimates below control the energy
gap and the nonlocal tails arising from the interaction with the exterior
datum.

\begin{definition}[Frational $p$-harmonic replacement]\label{def:harmonic-replacement}
	Let \(B\subset\mathbb{R}^n\) be a bounded open set, and let
	\[
	u\in \mathrm{W}^{s,p}_{\mathrm{loc}}(\mathbb{R}^n)
	\cap \mathrm{L}^{p-1}_{sp}(\mathbb{R}^n).
	\]
	We say that \(v\) is the fractional \(p\)-harmonic replacement of
	\(u\) in \(B\) if $v-u\in \mathrm{W}^{s,p}_0(B)$ and
	\[
	(-\Delta)^s_p v=0
	\quad\text{weakly in }B.
	\]
\end{definition}

We next record a simple tail estimate for harmonic replacements. The point is
that the tail of the replacement at scale \(R\) splits into an annular part,
where the function may have changed, and a far-away part, where the
replacement coincides with the original function.

\begin{lemma}[Tail estimate for harmonic replacements]\label{TailLema}
	Let \(p>1\), \(s\in(0,1)\), and let $u\in \mathrm{L}^{p-1}_{sp}(\mathbb{R}^n)$. Let \(v\) be the fractional $p$-harmonic replacement of \(u\) in
	\(B_{2R}(x_0)\). Then,
	\[
	\operatorname{Tail}(v;x_0,R)
	\leq
	c\, R^{-\frac np}
	\|v\|_{\mathrm{L}^p(B_{2R}(x_0)\setminus B_R(x_0))}
	+
	c\,\operatorname{Tail}(u;x_0,2R),
	\]
	where \(c>0\) depends only on \(n,p\), and \(s\).
\end{lemma}

\begin{proof}
	Since \(v=u\) a.e. in \(\mathbb{R}^n\setminus B_{2R}(x_0)\), we split
	the tail as
	\[
		\operatorname{Tail}(v;x_0,R)^{p-1}
		=
		R^{sp}
		\int_{B_{2R}(x_0)\setminus B_R(x_0)}
		\frac{|v(x)|^{p-1}}{|x-x_0|^{n+sp}}\,\dd x
		+
		R^{sp}
		\int_{\mathbb{R}^n\setminus B_{2R}(x_0)}
		\frac{|u(x)|^{p-1}}{|x-x_0|^{n+sp}}\, \dd x.
	\]
	We estimate each term of the right-hand side separately, the annular contribution first. Since
	\(|x-x_0|\geq R\) in \(B_{2R}(x_0)\setminus B_R(x_0)\), Hölder's
	inequality gives
	\[
	\begin{split}
		\int_{B_{2R}(x_0)\setminus B_R(x_0)}
		\frac{|v(x)|^{p-1}}{|x-x_0|^{n+sp}}\,\dd x
		&\leq
		R^{-n-sp}
		\int_{B_{2R}(x_0)\setminus B_R(x_0)}
		|v(x)|^{p-1}\,\dd x
		\\[0.5em]
		&\leq
		c\, R^{-n-sp+\frac np}
		\|v\|_{\mathrm{L}^p(B_{2R}(x_0)\setminus B_R(x_0))}^{p-1}.
	\end{split}
	\]
	For the exterior contribution, we observe that, by the definition of the tail at scale
	\(2R\),
	\[
	\begin{aligned}
		R^{sp}
		\int_{\mathbb{R}^n\setminus B_{2R}(x_0)}
		\frac{|u(x)|^{p-1}}{|x-x_0|^{n+sp}}\,\dd x
		=
		2^{-sp}\operatorname{Tail}(u;x_0,2R)^{p-1}.
	\end{aligned}
	\]
	Combining these estimates with the elementary estimate
	$(a+b)^{\frac{1}{p-1}}
	\leq
	c_p(a^{\frac{1}{p-1}}+b^{\frac{1}{p-1}})$,
	the proof is complete.
\end{proof}

The estimate above uses only the coincidence \(v=u\) in \(\mathbb{R}^n\setminus B_{2R}(x_0)\), and not the fact that \(v\) is a harmonic replacement of $u$. Although the replacement modifies the function inside \(B_{2R}(x_0)\),
its tail at the smaller scale \(R\) is affected only through the intermediate
annulus. The contribution from the far exterior is exactly inherited from the
original datum.

The next lemma quantifies the energy decrease produced by the harmonic
replacement. It says that the loss of energy from $u$ to its replacement
controls the fractional seminorm of their difference.

\begin{lemma}[Energy gap for the harmonic replacement]\label{lemaTFC}
	Let $p\geq 2$, $s\in(0,1)$, and let $B\subset\mathbb{R}^n$ be a
	bounded open set. Let $u$ be admissible, and let $v$ be the
	fractional $p$-harmonic replacement of $u$ in $B$. Then, there exists
	a constant $c>0$, depending only on $p$, such that
	\[
	\int\!\!\int_{Q_B}
	\frac{|u(x)-u(y)|^p-|v(x)-v(y)|^p}{|x-y|^{n+sp}}
	\,\dd x\,\dd y
	\geq
	c
	\int\!\!\int_{Q_B}
	\frac{|(u-v)(x)-(u-v)(y)|^p}{|x-y|^{n+sp}}
	\,\dd x\,\dd y,
	\] where $Q_B := (B^c\times B^c)^c$.
\end{lemma}

\begin{proof}
	For $\tau\in[0,1]$, we denote $u_\tau:=\tau u+(1-\tau)v$ and apply the fundamental theorem of calculus. Since $v$ is the fractional $p$-harmonic replacement of $u$ in $B$,
	we have $w:=u-v\in \mathrm{W}^{s,p}_0(B)$ and
	\begin{equation}\label{eqn:harm-repl-soln}
		\int\!\!\int_{Q_B}
		\frac{
			j_p\big(v(x)-v(y)\big)\big(w(x)-w(y)\big)
		}
		{|x-y|^{n+sp}}
		\,\dd x\,\dd y
		=0.
	\end{equation}
	By the fundamental theorem of calculus, we have
	\[
	\begin{split}
		\mathbb{I}_{u,v} &:= \int\!\!\int_{Q_B}
		\frac{|u(x)-u(y)|^p-|v(x)-v(y)|^p}{|x-y|^{n+sp}}
		\,\dd x\,\dd y
		\\
		& =
		p\int_0^1
		\int\!\!\int_{Q_B}
		\frac{
			j_p\big(u_\tau(x)-u_\tau(y)\big)\big(w(x)-w(y)\big)
		}
		{|x-y|^{n+sp}}
		\,\dd x\,\dd y\,\dd\tau .
	\end{split}
	\]
	Using \eqref{eqn:harm-repl-soln}, we
	obtain
	\[
		\mathbb{I}_{u,v} =
		p\int_0^1
		\int\!\!\int_{Q_B}
		\frac{
			\left[
			j_p(u_\tau(x)-u_\tau(y))-j_p(v(x)-v(y))
			\right]
			\big(w(x)-w(y)\big)
		}
		{|x-y|^{n+sp}}
		\,\dd x\,\dd y\,\dd\tau.
	\]
	For $\tau>0$, observe
	\[
	w(x)-w(y)
	=
	\frac{
		\big(u_\tau(x)-v(x)\big)-\big(u_\tau(y)-v(y)\big)
	}{\tau};
	\]
	hence, together with the elementary monotonicity inequality
	\[
	\bigl(j_p(a)-j_p(b)\bigr)(a-b)
	\geq
	2^{2-p}|a-b|^p,
	\quad a,b\in\mathbb{R},
	\]
	we have
	\[
	\mathbb{I}_{u,v} \geq
		c
		\int_0^1
		\tau^{-1}
		\int\!\!\int_{Q_B}
		\frac{
			\big|(u_\tau-v)(x)-(u_\tau-v)(y)\big|^p
		}
		{|x-y|^{n+sp}}
		\,\dd x\,\dd y\,\dd\tau.
	\]
	Finally, $u_\tau-v=\tau(u-v)=\tau w=\tau (u-v)$, we conclude
	\[
	\mathbb{I}_{u,v} \geq
		c
		\int_0^1 \tau^{p-1}\,\dd\tau
		\int\!\!\int_{Q_B}
		\frac{|(u-v)(x)-(u-v)(y)|^p}{|x-y|^{n+sp}}
		\,\dd x\,\dd y. \qedhere
	\]
\end{proof}

The next elementary estimate allows us to change the center of a tail, provided
the new center remains in a smaller concentric ball. The price is a local
\(\mathrm{L}^{\infty}\)-term and the tail at the original center with the
larger radius.

\begin{lemma}[Change of center for tails]\label{Tail estimate1}
	Let $0<R_1<R_2$, fix $x_0\in\mathbb{R}^n$, and let $u\in \mathrm{L}^{\infty}(B_{R_2}(x_0))
	\cap
	\mathrm{L}^{p-1}_{sp}(\mathbb{R}^n)$.
	Then, for every $0<R<R_2-R_1$ and every $y\in B_{R_1}(x_0)$, we have
	\[
	\operatorname{Tail}(u;y,R)
	\leq
	c\, \Big(\|u\|_{\mathrm{L}^{\infty}(B_{R_2}(x_0))}
	+
	\operatorname{Tail}(u;x_0,R_2)\Big),
	\]
	where $c>0$ depends only on $n,p,s$ and $R_2/(R_2-R_1)$.
\end{lemma}

\begin{proof}
	We split the tail into the contribution coming from the bounded region
	$B_{R_2}(x_0)$ and from its complement. Namely,
	\[
		\operatorname{Tail}(u;y,R)^{p-1}
		=
		R^{sp}
		\int_{B_{R_2}(x_0)\setminus B_R(y)}
		\frac{|u(x)|^{p-1}}{|x-y|^{n+sp}}
		\,\dd x +
		R^{sp}
		\int_{\mathbb{R}^n\setminus B_{R_2}(x_0)}
		\frac{|u(x)|^{p-1}}{|x-y|^{n+sp}}
		\,\dd x.
	\]
	For the first term, since $y\in B_{R_1}(x_0)$ and $R_1<R_2$, we have $B_{R_2}(x_0)\subset B_{2R_2}(y)$; thus,
	\[
	\begin{split}
		R^{sp}
		\int_{B_{R_2}(x_0)\setminus B_R(y)}
		\frac{|u(x)|^{p-1}}{|x-y|^{n+sp}}
		\,\dd x
		&\leq
		R^{sp}
		\|u\|_{\mathrm{L}^{\infty}(B_{R_2}(x_0))}^{p-1}
		\int_{B_{2R_2}(y)\setminus B_R(y)}
		\frac{1}{|x-y|^{n+sp}}
		\,\dd x
		\\[0.5em]
		&\leq
		C\|u\|_{\mathrm{L}^{\infty}(B_{R_2}(x_0))}^{p-1}.
	\end{split}
	\]
	For the second term, if $x\in\mathbb{R}^n\setminus B_{R_2}(x_0)$
	and $y\in B_{R_1}(x_0)$, then
	\[
	|x-y|
	\geq
	|x-x_0|-|y-x_0|
	\geq
	|x-x_0|-R_1
	\geq
	\frac{R_2-R_1}{R_2}|x-x_0|.
	\]
	Since $R<R_2-R_1<R_2$, we obtain
	\[
	\begin{split}
		R^{sp}
		\int_{\mathbb{R}^n\setminus B_{R_2}(x_0)}
		\frac{|u(x)|^{p-1}}{|x-y|^{n+sp}}
		\,\dd x
		\leq
		C
		R_2^{sp}
		\int_{\mathbb{R}^n\setminus B_{R_2}(x_0)}
		\frac{|u(x)|^{p-1}}{|x-x_0|^{n+sp}}
		\,\dd x=
		C\operatorname{Tail}(u;x_0,R_2)^{p-1}.
	\end{split}
	\]
	Combining these two estimates and taking the power $1/(p-1)$, the proof is complete.
\end{proof}

Thus, tails are stable under changes of center away from the boundary of the
larger ball. This allows us to estimate tails centered at arbitrary points
inside \(B_{R_1}(x_0)\) by a fixed tail centered at \(x_0\).

\section{Existence and first properties of minimizers}\label{sec:existence-first-properties}

In this section, we establish the existence of minimizers for the functional
\eqref{main functional} and prove the elementary consequences of minimality of \Cref{thm:intro-existence}.
These facts provide the variational starting point for the regularity theory
developed in the sequel.

We begin with the compactness part of the variational theory. The direct method
applies to the fractional energy once the exterior datum is fixed, since the
only non-fixed contribution outside $\Omega$ is the interaction between
$\Omega$ and its complement.

\begin{proposition}
	Let $\Omega\subset\mathbb{R}^n$ be a bounded domain and let
	$g\in \mathrm{W}^{s,p}(\mathbb{R}^n)$, with
	$g\geq 0$ a.e. in $\mathbb{R}^n\setminus\Omega$. Then, there exists a
	minimizer of $\mathcal{I}_M$ over the class $\mathcal{A}_{g}$.
\end{proposition}

\begin{proof}
	Let $\{u_k\}_{k\in\mathbb{N}}\subset \mathrm{W}^{s,p}(\mathbb{R}^n)$, with that $u_k=g$ a.e. in
	$\mathbb{R}^n\setminus\Omega$, be a
	minimizing sequence. Since $g$ is an admissible competitor, there
	exists a constant $C>0$ such that, for every $k\in\mathbb{N}$,
	\[
	[u_k]_{\mathrm{W}^{s,p}(\mathbb{R}^n)}^p
	+
	M|\{u_k>0\}\cap\Omega|
	\leq C
	\]
	Since $u_k-g=0$ a.e. in $\mathbb{R}^n\setminus\Omega$, the fractional
	Poincaré--Friedrichs inequality, see \cite[Lemma 4.7]{C17}, gives
	\[
	\|u_k-g\|_{\mathrm{L}^p(\Omega)}
	\leq
	C [u_k-g]_{\mathrm{W}^{s,p}(\mathbb{R}^n)}
	\leq
	C\big(
	[u_k]_{\mathrm{W}^{s,p}(\mathbb{R}^n)}
	+
	[g]_{\mathrm{W}^{s,p}(\mathbb{R}^n)}
	\big)
	\leq C.
	\]
	Hence $(u_k-g)_{k\in\mathbb{N}}$ is bounded in
	$\mathrm{W}^{s,p}(\Omega)$. By the compact embedding theorem for
	fractional Sobolev spaces \cite[Theorem 7.1]{NPV12}, up to a subsequence,
	there exists $w\in \mathrm{L}^p(\Omega)$ such that
	\[
	u_k-g\to w
	\quad\text{in }\mathrm{L}^p(\Omega)
	\quad
	\text{and}
	\quad
	u_k-g\to w
	\quad\text{a.e. in }\Omega.
	\]
	Extending $w$ by zero to $\mathbb{R}^n\setminus\Omega$, define
	\[
	u:=g+w.
	\]
	Then $u=g$ a.e. in $\mathbb{R}^n\setminus\Omega$, and
	$u_k\to u$ a.e. in $\mathbb{R}^n$.
	By Fatou's lemma,
	\[
	[u]_{\mathrm{W}^{s,p}(\mathbb{R}^n)}^p
	\leq
	\liminf_{k\to\infty} \,
	[u_k]_{\mathrm{W}^{s,p}(\mathbb{R}^n)}^p.
	\]
	Moreover, since $u_k\to u$ a.e. in $\Omega$, we have
	\[
	|\{u>0\}\cap\Omega|
	\leq
	\liminf_{k\to\infty}
	|\{u_k>0\}\cap\Omega|.
	\]
	Therefore, we have lower semicontinuity
	\[
	\mathcal{I}_M(u)
	\leq
	\liminf_{k\to\infty}
	\mathcal{I}_M(u_k).
	\]
	and $u$ is a minimizer of $\mathcal{I}_M$ over the class
	\eqref{minimizer}.
\end{proof}

The next lemma records the first variational consequence of minimality. Since
decreasing a minimizer by a nonnegative perturbation cannot increase the
positivity set, the energy variation alone yields a one-sided Euler--Lagrange
inequality. In addition, nonnegative exterior data force minimizers to be
nonnegative almost everywhere.

\begin{lemma}\label{weak subsolution}
	Let $u$ be a minimizer of $\mathcal{I}_{M}$ over the class
	\eqref{minimizer}, and assume that
	$g\geq 0$ a.e. in $\mathbb{R}^n\setminus\Omega$. Then $u$ is a weak
	subsolution of \eqref{maineq} in $\Omega$. Moreover,
	$u\geq 0$ a.e. in $\Omega$.
\end{lemma}

\begin{proof}
    Let $\varphi\in \mathrm{W}^{s,p}_0(\Omega)$ be nonnegative and let $t>0$. Since
			$\varphi=0$ a.e. in $\mathbb{R}^n\setminus\Omega$, the function
			$u-t\varphi$ is admissible class. Moreover,
			\[
			\{u-t\varphi>0\}\cap\Omega
			\subset
			\{u>0\}\cap\Omega.
			\] This inclusion, together with minimality,
			$\mathcal I_M(u)
			\leq
			\mathcal I_M(u-t\varphi)$, gives
			\[
			0
			\leq
			\frac{
				[u-t\varphi]_{\mathrm{W}^{s,p}(\mathbb{R}^n)}^p
				-
				[u]_{\mathrm{W}^{s,p}(\mathbb{R}^n)}^p
			}{t}.
			\]
			Letting $t\to 0^{+}$, we obtain
			\[
			0
			\leq
			-p
			\int_{\mathbb{R}^n}\int_{\mathbb{R}^n}
			\frac{
				j_p(u(x)-u(y))(\varphi(x)-\varphi(y))
			}
			{|x-y|^{n+sp}}
			\,\dd x\,\dd y.
			\]
			Hence $\mathcal E(u,\varphi)\leq0$, as desired.

	We now prove that $u\geq0$ a.e. in $\Omega$. Since
	$g\geq0$ a.e. in $\mathbb{R}^n\setminus\Omega$, the function $u_+$ is
	also admissible. Moreover, $\{u_+>0\}\cap\Omega
	=
	\{u>0\}\cap\Omega$ and the map $t\mapsto t_+$ is $1$-Lipschitz, we have
	\[
	[u_+]_{\mathrm{W}^{s,p}(\mathbb{R}^n)}^p
	\leq
	[u]_{\mathrm{W}^{s,p}(\mathbb{R}^n)}^p,
	\]
	and consequently $\mathcal{I}_M(u_+)
	\leq
	\mathcal{I}_M(u)$.
	
	If the set $\{u<0\}\cap\Omega$ had positive measure, then the above
	inequality would be strict because $u=g\geq0$ a.e. in
	$\mathbb{R}^n\setminus\Omega$, and the interaction between
	$\{u<0\}\cap\Omega$ and $\mathbb{R}^n\setminus\Omega$ would strictly
	decrease after replacing $u$ by $u_+$. This would contradict the
	minimality of $u$. Hence $u=u_+$ a.e. in $\Omega$, that is,
	$u\geq0$ a.e. in $\Omega$.
\end{proof}

As a consequence, minimizers are locally bounded. More precisely, the
subsolution estimate from \Cref{Boundedness} applies directly.

\begin{corollary}\label{Local bound to subsolution}
	Let $u$ be a minimizer of $\mathcal{I}_{M}$ over the class
	\eqref{minimizer}, with $g\geq0$ a.e. in
	$\mathbb{R}^n\setminus\Omega$. Then $u$ is locally bounded in
	$\Omega$. Moreover, if $B_R(x_0)\subset\Omega$, then
	\[
	\sup_{B_{R/2}(x_0)} u
	\leq
	\delta\operatorname{Tail}(u_+;x_0,R/2)
	+
	C\delta^{-\frac{(p-1)n}{sp^2}}
	\left(
	\mint_{B_R(x_0)} u_+^p\,\dd x
	\right)^{1/p}
	\]
	for every $\delta\in(0,1]$, where $C>0$ depends only on
	$n$, $p$, and $s$.
\end{corollary}

\begin{proof}
	By \Cref{weak subsolution}, $u$ is a weak subsolution to
	\eqref{maineq} in $\Omega$. The estimate follows directly from
	\Cref{Boundedness}.
\end{proof}

Thus, minimizers enjoy the basic compactness, sign, and local boundedness
properties needed in the sequel. In the next section, we use
these facts together with comparison estimates for fractional $p$-harmonic
replacements.

\begin{remark}
			After the local continuity of minimizers is known, the previous one-sided variational
			inequality from \Cref{weak subsolution} can be upgraded to the homogeneous equation inside the positivity set.
			Indeed, if $\operatorname{supp}\varphi\subset\subset \Omega\cap\{u>0\}$, then small
			two-sided perturbations $u+t\varphi$ do not change the positivity set. Thus, the
			minimality inequality may be tested with both signs of $t$, yielding $\mathcal E(u,\varphi)=0$; see the proof of \Cref{thm:intro-positive-phase} in the next section.
		\end{remark}

\section{Local Hölder regularity of minimizers}\label{Sec Holder regularity}

The aim of this section is to prove local Hölder regularity estimates for minimizers of $\mathcal{I}_M$. The argument combines the comparison between a minimizer and its fractional $p$-harmonic replacement with the auxiliary estimates from the previous section. The final step is based on the Campanato characterization of Hölder spaces.

The first step is to quantify the energy decrease produced by replacing a minimizer by its fractional $p$-harmonic replacement. Since the replacement agrees with the minimizer outside the ball, the only possible change in the volume comes from the positivity set inside the ball. This yields the following estimate.

\begin{lemma}\label{lem:replacement-energy-zero-phase}
	Let $u$ be a minimizer of $\mathcal{I}_M$ in $\Omega$. Let $B_{2R}(x_0)\subset\Omega$, and let $v$ be the fractional $p$-harmonic replacement of $u$ in $B_{2R}(x_0)$. Then, 
	\[
	\int\!\!\int_{\left(B_{2R}(x_0)^c\times B_{2R}(x_0)^c\right)^c}
	\frac{|(u-v)(x)-(u-v)(y)|^p}{|x-y|^{n+sp}}
	\,\dd x\,\dd y
	\leq
	M |\{u=0\}\cap B_{2R}(x_0)|.
	\]
	In particular, with a dimensional constant,
	\[
	\int\!\!\int_{\left(B_{2R}(x_0)^c\times B_{2R}(x_0)^c\right)^c}
	\frac{|(u-v)(x)-(u-v)(y)|^p}{|x-y|^{n+sp}}
	\,\dd x\,\dd y
	\leq
	c\, M R^n.
	\]
\end{lemma}

\begin{proof}
	Temporarily set $B=B_{2R}(x_0)$ and $Q_B := (B^c\times B^c)^c$. Since $v$ is the $p$-harmonic replacement of $u$ in $B$, we have $v=u$ in $\mathbb{R}^n\setminus B$. In particular, since $B\subset\Omega$, the function $v$ is an admissible competitor for $u$ in the minimization problem.
	By \Cref{lemaTFC}, applied in the set $Q_B$, we obtain, for $c=c(p)>0$,
	\[
	    \int\!\!\int_{Q_B}
	\frac{|(u-v)(x)-(u-v)(y)|^p}{|x-y|^{n+sp}}
	\,\dd x\,\dd y
	\leq \,\ 
	c
	\int\!\!\int_{Q_B}
	\frac{|u(x)-u(y)|^p - |v(x)-v(y)|^p}{|x-y|^{n+sp}}
	\,\dd x\,\dd y.
	\]
	
	Since $u=v$ in $\mathbb{R}^n\setminus B$, the difference of the global energies is supported in $Q_B$. Therefore,
	\[
	\int\!\!\int_{Q_B}
	\frac{|u(x)-u(y)|^p}{|x-y|^{n+sp}}
	\,\dd x\,\dd y
	-
	\int\!\!\int_{Q_B}
	\frac{|v(x)-v(y)|^p}{|x-y|^{n+sp}}
	\,\dd x\,\dd y
	\]
	\[
	=
	\int\!\!\int_{(\Omega^c\times \Omega^c)^c}
	\frac{|u(x)-u(y)|^p}{|x-y|^{n+sp}}
	\,\dd x\,\dd y
	-
	\int\!\!\int_{(\Omega^c\times \Omega^c)^c}
	\frac{|v(x)-v(y)|^p}{|x-y|^{n+sp}}
	\,\dd x\,\dd y.
	\]
	Using the minimality of $u$ and the admissibility of $v$, we have
	\[
	\int\!\!\int_{(\Omega^c\times \Omega^c)^c}
	\frac{|u(x)-u(y)|^p}{|x-y|^{n+sp}}
	\,\dd x\,\dd y
	+
	M|\{u>0\}\cap\Omega|
	\]
	\[
	\leq
	\int\!\!\int_{(\Omega^c\times \Omega^c)^c}
	\frac{|v(x)-v(y)|^p}{|x-y|^{n+sp}}
	\,\dd x\,\dd y
	+
	M|\{v>0\}\cap\Omega|.
	\]
	Hence,
	\[
	\int\!\!\int_{Q_B}
	\frac{|u(x)-u(y)|^p}{|x-y|^{n+sp}}
	\,\dd x\,\dd y
	-
	\int\!\!\int_{Q_B}
	\frac{|v(x)-v(y)|^p}{|x-y|^{n+sp}}
	\,\dd x\,\dd y
	\leq
	M\Big(
	|\{v>0\}\cap\Omega|
	-
	|\{u>0\}\cap\Omega|
	\Big).
	\]
	Since $u=v$ in $\mathbb{R}^n\setminus B$, the positivity sets of $u$ and $v$ agree outside $B$. Then, since minimizers are nonnegative, we have
	\[
	\begin{split}
	|\{v>0\}\cap\Omega|
	-
	|\{u>0\}\cap\Omega|
	& =
	|\{v>0\}\cap B|
	-
	|\{u>0\}\cap B|
	\\[0.2em]
	& \leq
	|B|-|\{u>0\}\cap B|
	\\[0.2em]
	& \leq
	|\{u=0\}\cap B|.
	\end{split}
	\]
	Combining the previous estimates, the proof is complete.
\end{proof}

Thus, whenever the zero phase of the minimizer has small measure inside a ball, the minimizer is close, in the nonlocal energy sense, to its fractional $p$-harmonic replacement. This estimate is combined with the known regularity of the fractional replacement to obtain a decay inequality for the oscillation of $u$.

We next transfer the oscillation decay of the fractional $p$-harmonic replacement to the original function $u$. The price paid in this comparison is measured by the local $\mathrm{L}^p$-distance between $u$ and its replacement.

\begin{lemma}\label{lem:oscillation-comparison-replacement}
	Let $u\in \mathrm{W}^{s,p}_{\mathrm{loc}}(\mathbb{R}^n)	\cap \mathrm{L}^{p-1}_{sp}(\mathbb{R}^n)$, and let $v$ be the fractional $p$-harmonic replacement of $u$ in $B_{2R}(x_0)\subset\Omega$ as in \Cref{def:harmonic-replacement}. Let $\alpha\in(0,1)$ be the exponent from \Cref{CKP}. Then, there exists a constant $C=C(n,p,s,\alpha)>0$ such that, for every $0<r\leq R$,
	\[
	\begin{split}
		\frac{1}{r^n}
		\int_{B_r(x_0)}
		|u-(u)_r|^p\,\dd x
		\leq & \
		C
		\left(\frac{r}{R}\right)^{p\alpha}
		\left[
		\operatorname{Tail}(u;x_0,2R)
		+
		\left(
		\mint_{B_{2R}(x_0)}|u|^p\,\dd x
		\right)^{1/p}
		\right]^p
		\\[0.3em]
		& +
		C
		\mint_{B_{2R}(x_0)}
		|u-v|^p\,\dd x
		+
		\frac{C}{r^n}
		\int_{B_R(x_0)}
		|u-v|^p\,\dd x.
	\end{split}
	\]
\end{lemma}

\begin{proof}
	Write $B_\rho=B_\rho(x_0)$. Since $u-(u)_r
	=
	v-(v)_r
	+
	(u-v)-((u-v))_r$, we obtain
	\[
	\int_{B_r}
	|u-(u)_r|^p\,\dd x
	\leq
	C
	\int_{B_r}
	|v-(v)_r|^p\,\dd x
	+
	C
	\int_{B_r}
	\left|
	(u-v)-\bigl((u-v)\bigr)_r
	\right|^p
	\,\dd x.
	\]
	By Jensen's inequality, $\|(u-v)_{r}\|_{\mathrm{L}^{p}(B_r)} \le \|u-v\|_{\mathrm{L}^{p}(B_r)}$;
	hence, since $r\leq R$, this gives
	\[
	\int_{B_r}
	|u-(u)_r|^p\,\dd x
	\leq
	C
	\int_{B_r}
	|v-(v)_r|^p\,\dd x
	+
	C
	\int_{B_R}
	|u-v|^p\,\dd x.
	\]
	
	We now apply the interior oscillation estimate from \Cref{campanato_estimate} for the fractional $p$-harmonic function $v$ in $B_{2R}$. Since $0<r\leq R$, we have
	\[
	\int_{B_r}
	|v-(v)_r|^p\,\dd x
	\leq
	C r^n
	\left(\frac{r}{R}\right)^{p\alpha}
	\left[
	\operatorname{Tail}(v;x_0,R)
	+
	\left(
	\mint_{B_{2R}}|v|^p\,\dd x
	\right)^{1/p}
	\right]^p.
	\]
	Since $v$ is the fractional $p$-harmonic replacement of $u$ in $B_{2R}(x_0)\subset\Omega$, \Cref{TailLema} and the triangle inequality imply
	\[
	\begin{split}
	  \operatorname{Tail}(v;x_0,R)
	+
	\left(
	\mint_{B_{2R}}|v|^p\,\dd x
	\right)^{1/p}
	\leq
	C 
	\left[
	\operatorname{Tail}(u;x_0,2R)
	+
	\left(
	\mint_{B_{2R}}|u|^p\,\dd x
	\right)^{1/p} \right. \\[0.5em] \quad \left. 
	+
	\left(
	\mint_{B_{2R}}|u-v|^p\,\dd x
	\right)^{1/p}
	\right].
	\end{split}
	\]
	Consequently,
	\[
	\begin{split}
		\int_{B_r}
		|u-(u)_r|^p\,\dd x
		\leq & \  
		C r^n
		\left(\frac{r}{R}\right)^{p\alpha}
		\left[
		\operatorname{Tail}(u;x_0,2R)
		+
		\left(
		\mint_{B_{2R}}|u|^p\,\dd x
		\right)^{1/p}
		\right]^p
		\\[0.5em]
		& +
		C r^n
		\left(\frac{r}{R}\right)^{p\alpha}
		\mint_{B_{2R}}
		|u-v|^p\,\dd x
		+
		C
		\int_{B_R}
		|u-v|^p\,\dd x.
	\end{split}
	\]
	Since $r\leq R$, we have $\left(r/R\right)^{p\alpha}\leq 1$. Thus, dividing by $r^n$, we obtain the desired estimate.
\end{proof}

This comparison estimate reduces the Hölder regularity of $u$ to two ingredients: the known oscillation decay for the fractional $p$-harmonic replacement and a quantitative control of the error $u-v$.

The previous comparison estimate becomes effective for minimizers because the error $u-v$ can be controlled by the free-boundary term. Indeed, applying a fractional Poincaré inequality to $u-v$ and using the energy estimate from \Cref{lem:replacement-energy-zero-phase}, we obtain the following oscillation estimate.

\begin{corollary}\label{remark estimate}
	Let $u$ be a minimizer of $\mathcal{I}_M$ in $\Omega$, and suppose that $B_{2R}(x_0)\subset\Omega$. Let $\alpha\in(0,1)$ be the exponent from the interior oscillation estimate for fractional $p$-harmonic functions. Then, there exists a constant $C=C(n,p,s,\alpha)>0$ such that, for every $0<r\leq R$,
	\[
	\begin{split}
		\frac{1}{r^n}
		\int_{B_r(x_0)}
		|u-(u)_r|^p\,\dd x
		\leq & \ 
		C
		\left(\frac{r}{R}\right)^{p\alpha}
		\left[
		\operatorname{Tail}(u;x_0,2R)
		+
		\left(
		\mint_{B_{2R}(x_0)}
		|u|^p\,\dd x
		\right)^{1/p}
		\right]^p
		\\[0.3em]
		& +
		CMR^{sp}
		+
		CM\frac{R^{n+sp}}{r^n}.
	\end{split}
	\]
\end{corollary}

\begin{proof}
	Let $v$ be the fractional $p$-harmonic replacement of $u$ in $B_{2R}(x_0)$, and set $w=u-v$. By \Cref{lem:oscillation-comparison-replacement}, it remains to estimate the two terms
	\[
	\mint_{B_{2R}(x_0)}
	|w|^p\,\dd x
	\qquad \text{and} \qquad 
	\int_{B_R(x_0)}
	|w|^p\,\dd x.
	\]
	Since $v$ is the fractional $p$-harmonic replacement of $u$ in $B_{2R}(x_0)$, we have $w=0$ in $\mathbb{R}^n\setminus B_{2R}(x_0)$ and, in particular, $w=0$ in $B_{4R}(x_0)\setminus B_{2R}(x_0)$. The latter set has a fixed positive proportion of the measure of $B_{4R}(x_0)$. Thus, by the fractional Poincaré inequality \cite[Lemma 4.7]{C17},
	\[
	\int_{B_{2R}(x_0)}
	|w|^p\,\dd x
	\leq
	\int_{B_{4R}(x_0)}
	|w|^p\,\dd x 
	\leq
	CR^{sp}
	\int\!\!\int_{B_{4R}(x_0)\times B_{4R}(x_0)}
	\frac{|w(x)-w(y)|^p}{|x-y|^{n+sp}}
	\,\dd x\,\dd y.
	\]
	Since $w=0$ in $\mathbb{R}^n\setminus B_{2R}(x_0)$, the last double integral is controlled by the interaction energy over
	\[
	Q_{B_{2R}(x_0)}
	:=
	\left(B_{2R}(x_0)^c\times B_{2R}(x_0)^c\right)^c.
	\]
	Hence, by \Cref{lem:replacement-energy-zero-phase},
	\[
	\int\!\!\int_{B_{4R}(x_0)\times B_{4R}(x_0)}
	\frac{|w(x)-w(y)|^p}{|x-y|^{n+sp}}
	\,\dd x\,\dd y
	\leq
	\int\!\!\int_{Q_{B_{2R}(x_0)}}
	\frac{|w(x)-w(y)|^p}{|x-y|^{n+sp}}
	\,\dd x\,\dd y\leq
	CMR^n.
	\]
	Combining the last estimates, and since $B_R(x_0)\subset B_{2R}(x_0)$, we obtain
	\[
	\mint_{B_{2R}(x_0)}
	|w|^p\,\dd x
	\leq
	CMR^{sp}
	\qquad
	\text{and}
	\qquad
	\int_{B_R(x_0)}
	|w|^p\,\dd x
	\leq
	CMR^{n+sp}, 
	\]
	as desired.
\end{proof}

This estimate is the basic Campanato-type inequality for minimizers. The remaining task is to control the tail and the average term uniformly, and then choose the scale ratio in order to obtain the desired decay of the mean oscillation.

We now combine the oscillation estimate for minimizers with the Campanato characterization of Hölder spaces and finish the proof of \Cref{thm:intro-holder}. The key point is to choose the comparison scale as a power of the observation scale, balancing the decay inherited from the fractional $p$-harmonic replacement with the error produced by the free-boundary term.

		\begin{proof}[Proof of \Cref{thm:intro-holder}]
			Let $\alpha\in(0,1)$ be the exponent from \Cref{CKP}. We choose the exponent $\alpha_0$ as follows:
			\[
			\theta
			=
			\frac{n+p\alpha}{n+p(\alpha+s)}
			\qquad\text{and}\qquad
			\alpha_0
			=
			(1-\theta)\alpha.
			\]
			Then $0<\theta<1$ and
			\[
			\alpha_0
			=
			\frac{sp\alpha}{n+p(\alpha+s)}.
			\]
			Moreover, $(1-\theta)p\alpha
			=
			\theta(sp+n)-n
			=
			p\alpha_0$ and $p\alpha_0\leq \theta sp$.

			For the rest of the proof, let us write
			\[
			A_R
			=
			\operatorname{Tail}(u;x_0,2R)
			+
			\left(
			\mint_{B_{4R}(x_0)}
			|u|^p\,\dd x
			\right)^{1/p}.
			\]
			Fix $y\in B_R(x_0)$. Choose $c_*\in(0,1)$, depending only on $n,p,s$, and $\alpha$, such that $2c_*^\theta\leq 1$, and set $r_*=c_*R$. Let $0<r\leq r_*$. We choose the comparison scale
			\[
			\rho
			=
			R\left(\frac{r}{R}\right)^\theta.
			\]
			Observe that $0<\theta<1$ implies $r\leq\rho$. Moreover, since $r\leq c_*R$, our choice of $c_{*}$ gives
			\[
			2\rho
			=
			2R\left(\frac{r}{R}\right)^\theta
			\leq
			2R c_*^\theta
			\leq
			R.
			\]
			Then, $B_{2\rho}(y)\subset B_{2R}(x_0)\subset B_{4R}(x_0)\subset\Omega$ and we may apply \Cref{remark estimate} with center $y$ and radius $\rho$. We obtain
			\[
			\begin{split}
				\frac{1}{r^n}
				\int_{B_r(y)}
				|u-(u)_{y,r}|^p\,\dd x
				\leq {}&
				C
				\left(\frac{r}{\rho}\right)^{p\alpha}
				\left[
				\operatorname{Tail}(u;y,2\rho)
				+
				\left(
				\mint_{B_{2\rho}(y)}
				|u|^p\,\dd x
				\right)^{1/p}
				\right]^p
				\\[0.3em]
				&+
				CM\rho^{sp}
				+
				CM\frac{\rho^{sp+n}}{r^n}.
			\end{split}
			\]
			Since $y\in B_R(x_0)$ and $2\rho\leq R$, the change-of-center estimate for the tail, \Cref{Tail estimate1}, together with the local boundedness estimate for minimizers, \Cref{Boundedness}, gives
			\[
			\operatorname{Tail}(u;y,2\rho)
			+
			\left(
			\mint_{B_{2\rho}(y)}
			|u|^p\,\dd x
			\right)^{1/p}
			\leq
			C A_R;
			\] therefore,
			\[
			\frac{1}{r^n}
			\int_{B_r(y)}
			|u-(u)_{y,r}|^p\,\dd x
			\leq
			C
			\left(\frac{r}{\rho}\right)^{p\alpha}
			A_R^p
			+
			CM\rho^{sp}
			+
			CM\frac{\rho^{sp+n}}{r^n}.
			\]
			
			We now use the definition of $\rho$. Recall $p\alpha_0\leq\theta sp$; then,
			\[
			\left(\frac{r}{\rho}\right)^{p\alpha}
			=
			\left(\frac{r}{R}\right)^{p\alpha_0}
			\qquad 
			\text{and}
			\qquad
			\rho^{sp}
			=
			R^{sp}
			\left(\frac{r}{R}\right)^{\theta sp}
			\leq
			R^{sp}
			\left(\frac{r}{R}\right)^{p\alpha_0},
			\] and
			\[
			\frac{\rho^{sp+n}}{r^n}
			=
			R^{sp}
			\left(\frac{r}{R}\right)^{\theta(sp+n)-n}
			=
			R^{sp}
			\left(\frac{r}{R}\right)^{p\alpha_0}.
			\]
			Thus,
			\[
			\frac{1}{r^n}
			\int_{B_r(y)}
			|u-(u)_{y,r}|^p\,\dd x
			\leq
			C
			\left(\frac{r}{R}\right)^{p\alpha_0}
			\Big[
			A_R^p
			+
			MR^{sp}
			\Big].
			\]
			Equivalently,
			\[
			\left(
			\frac{1}{r^{n+p\alpha_0}}
			\int_{B_r(y)}
			|u-(u)_{y,r}|^p\,\dd x
			\right)^{1/p}
			\leq
			\frac{C}{R^{\alpha_0}}
			\left[
			A_R
			+
			M^{1/p}R^s
			\right].
			\]
			
			Since $y\in B_R(x_0)$ and $0<r\leq r_*$ are arbitrary, the local Campanato characterization of Hölder spaces yields
			\[
			[u]_{\mathrm{C}^{\alpha_0}(B_R(x_0))}
			\leq
			\frac{C}{R^{\alpha_0}}
			\left[
			A_R
			+
			M^{1/p}R^s
			\right].
			\] Substituting the definition of $A_R$, we obtain the desired estimate.
		\end{proof}

As a consequence of the local Hölder regularity, the positivity set $\Omega\cap\{u>0\}$ is open and  \Cref{thm:intro-positive-phase} follows. Indeed, variations compactly supported inside this set do not change the volume term in the functional, and the usual first variation of the energy gives the homogeneous equation.


\begin{proof}[Proof of \Cref{thm:intro-positive-phase}]
	By \Cref{thm:intro-holder}, the function $u$ is continuous in $\Omega$, so that $D:= \Omega\cap\{u>0\}$
	is open.
	Let $\phi\in \mathrm{C}_{\mathrm{c}}^\infty(D)$. Since $\operatorname{supp}\phi \subset \subset D$ and $u>0$ in $D$, there exists $m>0$ such that $u\geq m$ on $\operatorname{supp}\phi$.
	Choosing $t_0>0$ sufficiently small, we have $u+t\phi>0$ on $\operatorname{supp}\phi$, for every $|t|<t_0$. Since $\phi=0$ outside $\operatorname{supp}\phi$, it follows that, for every $|t|<t_0$,
	\[
	\{u+t\phi>0\}\cap\Omega
	=
	\{u>0\}\cap\Omega.
	\]

	Moreover, since $\phi\in \mathrm{C}_{\mathrm{c}}^\infty(D)\subset \mathrm{W}^{s,p}_0(\Omega)$, the function $u+t\phi$ is an admissible competitor for $u$. By the minimality of $u$, and using the equality of the positivity sets, we obtain
	\[
	\int\!\!\int_{(\Omega^c\times\Omega^c)^c}
	\frac{|u(x)-u(y)|^p}{|x-y|^{n+sp}}
	\,\dd x\,\dd y
	\leq
	\int\!\!\int_{(\Omega^c\times\Omega^c)^c}
	\frac{|u(x)+t\phi(x)-u(y)-t\phi(y)|^p}{|x-y|^{n+sp}}
	\,\dd x\,\dd y
	\]
	for every $|t|<t_0$. Thus $t=0$ is a minimum of the function
	\[
	F(t)
	:=
	\int\!\!\int_{(\Omega^c\times\Omega^c)^c}
	\frac{|u(x)+t\phi(x)-u(y)-t\phi(y)|^p}{|x-y|^{n+sp}}
	\,\dd x\,\dd y
	\]
	and so $F'(0)=0$. Computing the derivative at $t=0$, we obtain
	\[
	\int\!\!\int_{(\Omega^c\times\Omega^c)^c}
	\frac{|u(x)-u(y)|^{p-2}(u(x)-u(y))(\phi(x)-\phi(y))}
	{|x-y|^{n+sp}}
	\,\dd x\,\dd y
	=
	0,
	\] which is precisely the weak formulation of $(-\Delta)^s_p u=0$ in $D=\Omega\cap\{u>0\}$.
\end{proof}

\section{Optimal growth at free boundary points}\label{sec:optimal-growth}

In this section, we prove the optimal growth of minimizers at free boundary points. The main difficulty, compared with the local Alt--Caffarelli problem, is that the nonlocal operator does not allow a purely local flatness estimate: the values of the minimizer outside the ball enter through tail contributions. We first establish a flatness lemma for small values of the penalization parameter $M$, in which the nonlocal contribution is explicitly measured by the tail. This lemma is then used in an iterative argument to obtain the sharp growth rate.

In the linear case $p=2$, the optimal growth at free boundary points is of order $r^s$, corresponding to the sharp $\mathrm{C}^{s}$ regularity threshold for the fractional Alt--Caffarelli problem; see \cite{CRS10,RO24,ST25}. The result below shows that the same growth rate persists for the fractional $p$-Laplacian problem with $p\geq 2$.

For the iteration argument below, we fix a constant depending only on $n,s,p$, chosen so as to absorb the annular contribution in the tail estimate in the proof of \Cref{thm:intro-growth} below. More precisely, we choose $c_{n,s,p}>0$ so that
\begin{equation}\label{constant-iteration}
    \left(\frac{n\omega_n}{sp}\right)^{1/(p-1)}
    +
    c_{n,s,p}10^{-\frac{s}{p-1}}
    \leq
    c_{n,s,p}.
\end{equation}
The precise value of the constant is irrelevant for our purposes; only its structural dependence on $n,s$, and $p$ is used in the argument.

\begin{lemma}\label{Flatness lemma}
    For every $\varepsilon>0$, there exists $M_0>0$, depending only on $n,s,p$ and $\varepsilon$, such that the following holds. Let $u\geq0$ be a minimizer of $\mathcal{I}_M$ in $B_1$, with $M\leq M_0$, $u(0)=0$, and
    \begin{equation}\label{eqn:smallness-lem}
        \mint_{B_1} u^p\,\dd x\leq 1.
    \end{equation}
    Then,
    \begin{equation}\label{lem:sharp-growth}
        \sup_{B_{1/10}} u
        \leq
        \varepsilon
        +
        \frac{1}{2\cdot10^{2s}c_{n,s,p}}
        \operatorname{Tail}(u;0,1/2).
    \end{equation}
\end{lemma}

\begin{proof}
	Let $v$ be the $p$-harmonic replacement of $u$ in $B_{1/2}$. First, we claim that
	\begin{equation}\label{bounded M}
		\mint_{B_{1/2}} |u-v|^p\,\dd x
		\leq
		C M.
	\end{equation} To see this, set $w:=u-v$. Since $v=u$ in $\mathbb{R}^n\setminus B_{1/2}$, we have, in particular, $w=0$ in $B_1\setminus B_{1/2}$.
	Thus, by \cite[Lemma 4.7]{C17}, applied to $w$ in $B_1$, we obtain
	\[
	\int_{B_1}|w|^p\,\dd x
	\leq
	C
	\int_{B_1}\int_{B_1}
	\frac{|w(x)-w(y)|^p}{|x-y|^{n+sp}}
	\,\dd x\,\dd y
	\le C
	\int\!\!\int_{(B_{1/2}^c\times B_{1/2}^c)^c}
	\frac{|w(x)-w(y)|^p}{|x-y|^{n+sp}}
	\,\dd x\,\dd y.
	\]
	By \Cref{lem:replacement-energy-zero-phase}, we have
	\[
	\int\!\!\int_{(B_{1/2}^c\times B_{1/2}^c)^c}
	\frac{|w(x)-w(y)|^p}{|x-y|^{n+sp}}
	\,\dd x\,\dd y
	\leq
	M|\{u=0\}\cap B_{1/2}|.
	\]
	Hence, since $w=0$ in $B_1\setminus B_{1/2}$ and $|\{u=0\}\cap B_{1/2}| \le \omega_{n}2^{-n}$, these estimates give \eqref{bounded M}.
	
	\smallskip
	
	Since $u$ is a subsolution of $(-\Delta)^s_pu=0$ in $B_{1/2}$ and $v=u$ in $\mathbb{R}^n\setminus B_{1/2}$, the comparison principle gives $u\leq v$ in $B_{1/2}$.
	In particular,
	\[
	\sup_{B_{1/10}}u
	\leq
	\sup_{B_{1/10}}v.
	\] 
	Since $v\geq0$ in $\mathbb{R}^n$, we have $\operatorname{Tail}(v_-;0,R)=0$ for every $R>0$; then, by the nonlocal Harnack estimate of \Cref{Nonlocal Harnack inequality}, applied in a fixed ball contained in $B_{1/2}$, we obtain
	\[
	\sup_{B_{1/10}}u
	\leq
	C v(0).
	\]
	We can now estimate $v(0)$ by the local boundedness estimate. Let $\iota\in(0,1/4)$ be a constant to be chosen later. Since $0\in B_{\iota/2}$, \Cref{Boundedness} applied with $\delta = 1$ gives
	\begin{equation}\label{harnackBound}
		\sup_{B_{1/10}}u
		\leq
		C\operatorname{Tail}(v;0,\iota/2)
		+
		C
		\left(
		\mint_{B_{\iota}} |v|^p\,\dd x
		\right)^{1/p}.
	\end{equation}
	Thus, the proof reduces to estimating the two terms on the right-hand side of \eqref{harnackBound}.
	
	\smallskip
	
	For the last term, we use the triangle inequality and \eqref{bounded M} to estimate
	\[
	\mint_{B_{\iota}}|v|^p\,\dd x
	\le C \mint_{B_{\iota}} |u-v|^p\,\dd x + C \mint_{B_{\iota}} |u|^p\,\dd x \le
	C\iota^{-n}M
	+
	C\mint_{B_{\iota}}|u|^p\,\dd x.
	\]
	Since $u(0)=0$, the Hölder estimate from \Cref{thm:intro-holder}, applied in $B_{1/4}$, yields
	\[
		\mint_{B_{\iota}}|u|^p\,\dd x
		=
		\mint_{B_{\iota}}|u(x)-u(0)|^p\,\dd x
		\leq
		C\iota^{p\alpha_0}
		\left[
		\operatorname{Tail}(u;0,1/2)
		+
		\left(
		\mint_{B_1}|u|^p\,\dd x
		\right)^{1/p}
		+
		M^{1/p}
		\right]^p.
	\]
	Taking $M_0\leq1$ and using \eqref{eqn:smallness-lem}, we obtain
	\[
	\mint_{B_{\iota}}|u|^p\,\dd x
	\leq
	C\iota^{p\alpha_0}
	\Big[
	\operatorname{Tail}(u;0,1/2)+2 \,
	\Big]^p
	\]
	so that
	\begin{equation}\label{average estimate}
		\mint_{B_{\iota}}|v|^p\,\dd x
		\leq
		C\iota^{-n}M
		+
		C\iota^{p\alpha_0}
		\Big[
		\operatorname{Tail}(u;0,1/2)+2 \,
		\Big]^p.
	\end{equation}
	
	\smallskip
	
	Next, we estimate the first term in the right-hand side of \eqref{harnackBound}, the tail of $v$. By \Cref{Tail estimate01}, applied to the nonnegative $p$-harmonic function $v$, and then by the nonlocal Harnack inequality, again using that $\operatorname{Tail}(v_-;0,R)=0$, we have
	\[
	\begin{split}
		\operatorname{Tail}(v;0,\iota/2)
		\leq
		C\sup_{B_{\iota/2}}v
		\leq
		C\sup_{B_{\iota}}v
		\leq
		C\inf_{B_{\iota}}v
		\leq
		C
		\left(
		\mint_{B_{\iota}}|v|^p\,\dd x
		\right)^{1/p}.
	\end{split}
	\]
	Combining this estimate with \eqref{average estimate}, we find
	\begin{equation}\label{Tail estimate}
		\operatorname{Tail}(v;0,\iota/2)
		\leq
		C\iota^{-n/p}M^{1/p}
		+
		C\iota^{\alpha_0}
		\Big[
		\operatorname{Tail}(u;0,1/2)+2\,
		\Big].
	\end{equation}
	
	Substituting \eqref{average estimate} and \eqref{Tail estimate} into \eqref{harnackBound}, and setting
	$
	\beta=\tfrac{(p-1)n}{sp^2},
	$
	we obtain
	\[
	\begin{split}
		\sup_{B_{1/10}}u
		\leq
		C \iota^{-n/p}M^{1/p}
		+
		C\iota^{\alpha_0}
		\Big[
		\operatorname{Tail}(u;0,1/2)+2\,
		\Big].
	\end{split}
	\]
	Thus, after increasing $C$ if necessary,
	\[
	\begin{split}
		\sup_{B_{1/10}}u
		\leq
		C \iota^{-n/p}M^{1/p}
		+
		C \iota^{\alpha_0}
		+
		C \iota^{\alpha_0}
		\operatorname{Tail}(u;0,1/2).
	\end{split}
	\]
	
	We now choose the parameters. First, fix for instance $\delta=1$. Next, choose $\iota\in(0,1/4)$, depending only on $n,s,p$, and $\varepsilon$, such that
	\[
	C\iota^{\alpha_0}
	\leq
	\min\left\{
	\frac{\varepsilon}{2},
	\frac{1}{2\cdot10^{2s}c_{n,s,p}}
	\right\}.
	\]
	Finally, choose $M_0>0$, with $M_0\leq1$, such that
	\[
	C\iota^{-n/p}M_0^{1/p}
	\leq
	\frac{\varepsilon}{2}.
	\]
	Therefore, if $M\leq M_0$, the previous estimate gives \eqref{lem:sharp-growth}, which proves the lemma.
\end{proof}

The sharp growth estimate from \Cref{thm:intro-growth} is then obtained by iterating the flatness lemma at decreasing scales. The main point is to propagate, simultaneously, a decay estimate for the supremum and a compatible decay estimate for the tail.

\begin{proof}[Proof of \Cref{thm:intro-growth}]
	Let
	$
	\varepsilon:=10^{-s}/2,
	$
	and let $c_{n,s,p}>0$ and $M_0>0$ be the constants given by \Cref{Flatness lemma}. We choose
	\[
	\tilde r :=
	\left(\frac{M_0}{M}\right)^{1/(sp)}.
	\]
	Fix $0<r<\tilde r$ such that $B_{2r}(x_0)\subset\Omega$. 
	Define
	\[
	\tilde u(x)
	:=
	\tau \, u(x_0+rx),
	\]
	where
	\[
	\tau
	:=
	\left[
	\sup_{B_r(x_0)}u
	+
	\frac{\operatorname{Tail}(u;x_0,r/2)}{10^s c_{n,s,p}}
	+
	1
	\right]^{-1}.
	\]
	Then, $\tilde u$ is a nonnegative minimizer of $\mathcal{I}_{\widetilde M}$ in $B_1$, where
	$\widetilde M
	=
	\tau^p r^{sp}M$.
	Since $\tau\leq1$ and $r<\tilde r$, we have
	$\widetilde M
	\leq
	r^{sp}M
	\leq
	M_0$.
	Moreover, $\tilde u(0)=0$, $\sup{\!}_{B_1}\tilde u\leq1$,
	and, in particular,
	\[
	\left(
	\mint_{B_1}\tilde u^p\,\dd x
	\right)^{1/p}
	\leq
	1.
	\] Furthermore, by the scaling of the tail,
	\begin{equation}\label{eqn:tailk=0}
		\operatorname{Tail}(\tilde u;0,1/2)
		=
		\tau\operatorname{Tail}(u;x_0,r/2)
		\leq
		10^s c_{n,s,p}.
	\end{equation}
	
	We claim that, for every integer $k\geq0$,
	\begin{equation}\label{induction1}
		\sup_{B_{10^{-k}}}\tilde u
		\leq
		10^{-ks},
	\end{equation}
	and
	\begin{equation}\label{induction2}
		\operatorname{Tail}(\tilde u;0,10^{-k}/2)
		\leq
		c_{n,s,p}10^{-(k-1)s}.
	\end{equation}
	The case $k=0$ follows from the definition of $\tau$ and from the estimate \eqref{eqn:tailk=0}.
	Assume now that \eqref{induction1} and \eqref{induction2} hold for some $k\geq0$ and define
	\[
	\tilde u_k(x)
	:=
	10^{ks}\tilde u(10^{-k}x).
	\]
	Then, $\tilde u_k$ is a nonnegative minimizer of $\mathcal{I}_{M_k}$ in $B_1$, with
	\[
	M_k
	=
	10^{ksp}10^{-ksp}\widetilde M
	=
	\widetilde M
	\leq
	M_0.
	\]
	Also, $\tilde u_k(0)=0$. From \eqref{induction1},
	\[
	\mint_{B_1}\tilde u_k^p\,\dd x
	=
	10^{ksp}
	\mint_{B_{10^{-k}}}\tilde u^p\,\dd x
	\leq
	1.
	\]
	Moreover, by the scaling of the tail and \eqref{induction2},
	\[
	\begin{split}
		\operatorname{Tail}(\tilde u_k;0,1/2)
		=
		10^{ks}
		\operatorname{Tail}(\tilde u;0,10^{-k}/2) \leq
		c_{n,s,p}10^s.
	\end{split}
	\]
	Applying \Cref{Flatness lemma} to $\tilde u_k$, with $\varepsilon=10^{-s}/2$, we obtain
	\[
		\sup_{B_{1/10}}\tilde u_k
		\leq
		\frac{10^{-s}}{2}
		+
		\frac{1}{2\cdot10^{2s}c_{n,s,p}}
		\operatorname{Tail}(\tilde u_k;0,1/2) \leq
		\frac{10^{-s}}{2}
		+
		\frac{10^{-s}}{2}
		=
		10^{-s}.
	\]
	Rescaling back, this gives
	\[
	\sup_{B_{10^{-(k+1)}}}\tilde u
	\leq
	10^{-(k+1)s}.
	\]
	This proves the induction step for \eqref{induction1} and	
	it remains to prove the induction step for \eqref{induction2}. Denote $\rho_k:=10^{-k}/2$; then, using the definition of the tail, the estimate just proved, and the induction hypothesis, we have
	\[
	\begin{split}
		\operatorname{Tail}(\tilde u;0,\rho_{k+1})^{p-1}
		&=
		\rho_{k+1}^{sp}
		\int_{\mathbb{R}^n\setminus B_{\rho_{k+1}}}
		\frac{\tilde u(y)^{p-1}}{|y|^{n+sp}}\,\dd y
		\\
		&=
		\rho_{k+1}^{sp}
		\int_{B_{\rho_k}\setminus B_{\rho_{k+1}}}
		\frac{\tilde u(y)^{p-1}}{|y|^{n+sp}}\,\dd y
		+
		\rho_{k+1}^{sp}
		\int_{\mathbb{R}^n\setminus B_{\rho_k}}
		\frac{\tilde u(y)^{p-1}}{|y|^{n+sp}}\,\dd y,
	\end{split}
	\]
	Since $B_{\rho_k}\subset B_{10^{-k}}$, \eqref{induction1} gives
	\[
	\begin{split}
		\rho_{k+1}^{sp}
		\int_{B_{\rho_k}\setminus B_{\rho_{k+1}}}
		\frac{\tilde u(y)^{p-1}}{|y|^{n+sp}}\,\dd y
		&\leq
		10^{-ks(p-1)}
		\rho_{k+1}^{sp}
		\int_{B_{\rho_k}\setminus B_{\rho_{k+1}}}
		\frac{1}{|y|^{n+sp}}\,\dd y \leq
		\frac{n\omega_n}{sp}
		10^{-ks(p-1)}.
	\end{split}
	\]
	On the other hand,
	\[
	\begin{split}
		\rho_{k+1}^{sp}
		\int_{\mathbb{R}^n\setminus B_{\rho_k}}
		\frac{\tilde u(y)^{p-1}}{|y|^{n+sp}}\,\dd y
		&=
		\left(\frac{\rho_{k+1}}{\rho_k}\right)^{sp}
		\operatorname{Tail}(\tilde u;0,\rho_k)^{p-1} =
		10^{-sp}
		\operatorname{Tail}(\tilde u;0,\rho_k)^{p-1}.
	\end{split}
	\]
	Since $p\geq2$, we have $1/(p-1)\leq1$, and hence the elementary inequality
	\[
	(a+b)^{1/(p-1)}
	\leq
	a^{1/(p-1)}+b^{1/(p-1)}
	\quad\text{for } a,b\geq0.
	\]
	Thus, by combining these, we obtain
	\[
	\begin{split}
		\operatorname{Tail}(\tilde u;0,\rho_{k+1})
		&\leq
		\left(\frac{n\omega_n}{sp}\right)^{1/(p-1)}
		10^{-ks} +
		10^{-\frac{sp}{p-1}}
		\operatorname{Tail}(\tilde u;0,\rho_k).
	\end{split}
	\]
	By the induction hypothesis,
	\[
	\begin{split}
		\operatorname{Tail}(\tilde u;0,\rho_{k+1})
		&\leq
		\left(\frac{n\omega_n}{sp}\right)^{1/(p-1)}
		10^{-ks} +
		c_{n,s,p}
		10^{-\frac{sp}{p-1}}
		10^{-(k-1)s}
		\\[0.3em]
		&=
		\left[
		\left(\frac{n\omega_n}{sp}\right)^{1/(p-1)}
		+
		c_{n,s,p}10^{-\frac{s}{p-1}}
		\right]
		10^{-ks}.
	\end{split}
	\] Now, the choice \eqref{constant-iteration} yields precisely \eqref{induction2} with $k+1$ in place of $k$. The induction is then complete.
	
	\smallskip
	
	We now derive the announced pointwise growth estimate for $\tilde u$. Let $x\in B_{1/2}$. If $x=0$, there is nothing to prove. Otherwise, choose $k\geq0$ such that
	\[
	10^{-(k+1)}
	<
	|x|
	\leq
	10^{-k}.
	\]
	By \eqref{induction1},
	\[
	\tilde u(x)
	\leq
	\sup_{B_{10^{-k}}}\tilde u
	\leq
	10^{-ks}
	=
	10^s10^{-(k+1)s}
	\leq
	10^s|x|^s.
	\]
	Thus, we have, for every $x\in B_{1/2}$,
	\[
	\tilde u(x)
	\leq
	C|x|^s.
	\]
	Back to the original variables, we obtain that, for every $x\in B_{r/2}(x_0)$,
	\[
	u(x)
	\leq
	C\tau^{-1}
	\frac{|x-x_0|^s}{r^s}.
	\]
	It remains to estimate $\tau^{-1}$. First, observe that
	\[
	\operatorname{Tail}(u;x_0,r/2)
	\leq
	C \, \Big(
	\sup_{B_r(x_0)}u
	+
	\operatorname{Tail}(u;x_0,r) \,
	\Big).
	\]
	Moreover, since minimizers are weak subsolutions, we may apply \Cref{Boundedness} with $R=2r$ and $\delta=1$. Since $u\geq0$, this yields
	\[
	\sup_{B_r(x_0)}u
	\leq
	C
	\left[
	\operatorname{Tail}(u;x_0,r)
	+
	\left(
	\mint_{B_{2r}(x_0)}u^p\,\dd x
	\right)^{1/p}
	\right].
	\]
	Combining the previous two estimates, we obtain
	\[
	\tau^{-1}
	\leq
	C
	\left(
	\operatorname{Tail}(u;x_0,r)
	+
	\left(
	\mint_{B_{2r}(x_0)}u^p\,\dd x
	\right)^{1/p}
	+
	1
	\right).
	\]
	This ends the proof.
\end{proof}

{\noindent{\bf Acknowledgments.} R. Costa was partly supported by Ph.D. fellowships from the Conselho Nacional de Desenvolvimento Científico e Tecnológico (CNPq), Brazil. This study was financed in part by the Coordenação de Aperfeiçoamento de Pessoal de Nível Superior - Brasil (CAPES) - Finance Code 001.}

%

\end{document}